\documentclass[12pt,a4paper]{amsart}
\pdfoutput=1

\usepackage[utf8]{inputenc}
\usepackage[T1]{fontenc}
\usepackage{lmodern}
\usepackage[kerning, spacing, tracking]{microtype}

\usepackage[style=numams, 
            sorting=nyt, 
            isbn=false,  
            backref=true, 
            backend=biber,
           ]{biblatex}
\usepackage{amsmath,amssymb, amsthm}
\usepackage{mathtools}
\usepackage{enumitem}

\usepackage[colorlinks, 
            linkcolor=blue, 
            citecolor=blue, 
            pdfa%
           ]{hyperref}

\usepackage[hcentering, hscale=0.7, vscale=0.79 ]{geometry}

\newtheorem*{mainthm}{Main Theorem}

\swapnumbers
\newtheorem{thm}{Theorem}[section]
\newtheorem{cor}[thm]{Corollary}
\newtheorem{lemma}[thm]{Lemma}
\newtheorem{obs}[thm]{Observation}
\theoremstyle{definition}
\newtheorem{example}[thm]{Example}
\theoremstyle{remark}
\newtheorem{remark}[thm]{Remark}

\newlist{enumthm}{enumerate}{1}  
\setlist[enumthm,1]{label=\textup{(\roman*)}}

\renewcommand{\phi}{\varphi}
\renewcommand{\theta}{\vartheta}

\renewcommand{\geq}{\geqslant}
\renewcommand{\leq}{\leqslant}

\newcommand{\nbd}{\nobreakdash-\hspace{0pt}}  
\newcommand{\defemph}[1]{\textbf{#1}} 

\newcommand{\nats}{\mathbb{N}}
\newcommand{\ints}{\mathbb{Z}}
\newcommand{\reals}{\mathbb{R}}
\newcommand{\compl}{\mathbb{C}}
\newcommand{\quats}{\mathbb{H}}

\newcommand{\eps}{\varepsilon}
\newcommand{\iso}{\cong}    
\newcommand{\into}{\hookrightarrow}  

\newcommand{\dblcs}[3]{#1 \setminus #2 / #3} 

\DeclarePairedDelimiter{\abs}{\lvert}{\rvert}
\DeclarePairedDelimiter{\norm}{\lVert}{\rVert}
\DeclarePairedDelimiter{\erz}{\langle}{\rangle}
\DeclarePairedDelimiter{\ipcf}{\lbrack}{\rbrack}  

\DeclarePairedDelimiter{\scp}{\langle}{\rangle} 
\newcommand{\lipr}[1]{\scp{#1}_{\Lambda}}        
\DeclareMathOperator{\Ker}{Ker}      

\DeclareMathOperator{\GL}{GL}
\DeclareMathOperator{\SL}{SL}
\DeclareMathOperator{\PSL}{PSL}
\DeclareMathOperator{\PGL}{PGL}
\DeclareMathOperator{\Og}{\mathbf{O}}  

\DeclareMathOperator{\Z}{\mathbf{Z}}         
\DeclareMathOperator{\Irr}{Irr}
\DeclareMathOperator{\Fix}{Fix}
\DeclareMathOperator{\ord}{ord}
\DeclareMathOperator{\mat}{\mathbf{M}}      
\DeclareMathOperator{\enmo}{End} 

\DeclareMathOperator{\tr}{tr}
\DeclareMathOperator{\id}{id}       

\DeclareMathOperator{\RC}{\mathcal{RC}}

\hypersetup{pdfauthor={Frieder Ladisch},
            pdftitle = {Realizations of abstract regular polytopes
                    from a representation theoretic view},
            pdfkeywords={Real representations of finite groups,
                         abstract regular polytope, 
                         realization cone, 
                         C-string group}
           }

\begin{document}
\title[Realizations of abstract regular polytopes%
      ]{Realizations of  
        abstract regular polytopes \\
        from a representation theoretic view%
       }
\author{Frieder Ladisch}
\thanks{Author supported by the DFG (Project: SCHU 1503/6-1)}
\address{Universität Rostock \\
         Institut für Mathematik \\
         18051 Rostock \\
         Germany}
\email{frieder.ladisch@uni-rostock.de}
%

\begin{abstract}
  Peter McMullen has developed a theory of 
  realizations of abstract regular polytopes,
  and has shown that the realizations up to congruence form
  a pointed convex cone which is the direct product of certain
  irreducible subcones.
  We show that each of these subcones is isomorphic to a 
  set of positive semi-definite hermitian matrices of dimension $m$
  over either the real numbers, 
  the complex numbers or the quaternions.
  In particular, we correct an erroneous computation of the 
  dimension of these subcones by McMullen and Monson.
  We show that the automorphism group of an abstract regular 
  polytope can have an irreducible character $\chi$ with
  $\chi\neq \overline{\chi}$ and with arbitrarily large
  essential Wythoff dimension.
  This gives counterexamples to a 
  result of Herman and Monson, which was derived from 
  the erroneous computation mentioned before. 
  
  We also discuss a relation between cosine vectors of certain
  pure realizations and the spherical functions appearing in the 
  theory of Gelfand pairs.
\end{abstract}
\maketitle

\section{Introduction}
These notes are the result of an attempt to understand 
realizations of abstract regular 
polytopes, as introduced by 
Peter 
McMullen~\cite{mcmullen89,mcmullenmonson03,mcmullen11,mcmullen14},
from a representation theoretic viewpoint,
thereby showing that the theory actually generalizes to a theory 
of ``realizations of transitive $G$-sets''.
That the theory applies in this wider context
was already pointed out by McMullen~\cite[Remark~2.1]{mcmullen14}. 
In particular, we will derive the exact structure
of McMullen's \emph{realization cone} using arguments from
basic representation theory and linear algebra.

To explain this in more detail,
and to state our main theorem, we have to introduce some notation.
Let 
$G$ be a finite group and $\Omega$ a transitive $G$\nbd set.
(In the original theory, $\Omega$ is the vertex set of an
abstract regular polytope and $G$ the automorphism group
of the polytope. But this assumption is in fact 
unnecessary for a large part of the theory.)
In this situation, one can define a closed pointed convex cone
called the \emph{realization cone}
which describes \emph{realizations} of the transitive $G$\nbd set 
$\Omega$ up to congruence.
(We will recall the exact definitions below.)

Let us write $\Irr_{\reals}G$ for the set of characters
of irreducible representations over the real numbers $\reals$.
McMullen~\cite{mcmullen89} has shown that the realization cone
is the direct product of subcones, each subcone corresponding
to some $\sigma\in \Irr_{\reals} G$
(or, what is the same, to a similarity class of 
irreducible representations of $G$).
We write $\RC_{\sigma}(\Omega)$ for the subcone
corresponding to $\sigma\in \Irr_{\reals}G$.
The main new result of this note concerns the structure of such a
subcone.

To state this result, we need some more notation.
Let $\pi=\pi_{\Omega}$ be the permutation character corresponding 
to the $G$\nbd set $\Omega$.
We can write $\pi$ as a sum of irreducible real characters:
\[ \pi = \sum_{\sigma\in \Irr_{\reals} G} m_{\sigma} \sigma.
\]
The \emph{multiplicities} $m_{\sigma}$ are uniquely determined by 
this equation, 
and equal the \emph{essential Wythoff dimension}
defined by McMullen and Monson~\cite{mcmullenmonson03}.
Moreover, to each $\sigma\in \Irr_{\reals}(G)$
belongs a division ring $\mathbb{D}_{\sigma}$
(the centralizer ring of a representation affording $\sigma$),
which is isomorphic to either 
the real numbers $\reals$,
the complex numbers $\compl$
or the Hamiltonian quaternions $\quats$.

We write $\mat_{m}(\mathbb{D})$ for the ring of 
$m\times m$-matrices over $\mathbb{D}$,
and if $A\in \mat_m(\mathbb{D})$, 
then $A^{*}$ denotes the (complex/quaternion) conjugate transpose
of $A$ when $\mathbb{D}=\compl$ or $\quats$,
and the transpose of $A$ when $\mathbb{D}=\reals$.
With this notation, we have:
\begin{mainthm}
  The realization cone of\/ $\Omega$ is 
  the direct product of subcones $\RC_{\sigma}(\Omega)$
  corresponding to $\sigma\in \Irr_{\reals}G$,
  where each $\RC_{\sigma}(\Omega)$ is isomorphic to
  the set of matrices
  \[ 
       \{ AA^{*}\mid A\in \mat_{m_{\sigma}}(\mathbb{D}_{\sigma})
       \}.
  \]
\end{mainthm}
In other words, the subcone $\RC_{\sigma}(\Omega)$ 
is isomorphic to the set of
hermitian positive semi-definite 
$m_{\sigma}\times m_{\sigma}$\nbd matrices
with entries in $\mathbb{D}_{\sigma}$,
with appropriate meaning of ``hermitian'' 
(depending on whether 
 $\mathbb{D}_{\sigma}= \reals$, $\compl$ or $\quats$).

From the main theorem, one can immediately derive the dimension
of $\RC_{\sigma}(\Omega)$ (see Corollary~\ref{c:dimsubcone}).
This dimension has been computed by
McMullen and Monson~\cite[Theorem~4.6]{mcmullenmonson03}
(using different notation).
Unfortunately, the result of McMullen and Monson
only matches with our description 
when $m_{\sigma}\leq 1$ or when $\mathbb{D}_{\sigma} = \reals$.
If the computation of 
McMullen and Monson~\cite[Theorem~4.6]{mcmullenmonson03}
were correct 
in the original situation, where
$G$ is the automorphism group of an abstract regular polytope
with vertex set $\Omega$, 
then it would follow that 
we always have $m_{\sigma}\leq 1$ or 
$\mathbb{D}_{\sigma}=\reals$ for such $G$.
And indeed, this is the main result of a paper
by Herman and Monson~\cite{hermanmonson04}.
They derive this from~\cite[Theorem~4.6]{mcmullenmonson03} 
in a different way. 

But unfortunately, the main result of
Herman and Monson~\cite{hermanmonson04} is wrong:
We show in Section~\ref{sec:counterex} that we can have
$\mathbb{D}_{\sigma} = \compl$ and $m_{\sigma}$ arbitrarily large
even when $G$ is the automorphism group of an abstract regular
polytope with vertex set~$\Omega$.
(See Example~\ref{expl:pslconcrete} for a concrete case where
 $m_{\sigma}=2$.
 It seems to be unknown whether there are any
 abstract regular polytopes with $\mathbb{D}_{\sigma}\iso \quats$
 for some $\sigma$.) 
These examples show that the computation of McMullen and Monson must
be wrong even in the original setting.
At the end of Section~\ref{sec:structure},
we briefly discuss where we see the flaw in McMullen's 
and Monson's proof.

A later result of McMullen~\cite[Theorem~5.2]{mcmullen14} can be 
interpreted as saying that the subcone
$\RC_{\sigma}(\Omega)$ is isomorphic to 
the symmetric positive semi-definite matrices of size 
$m_{\sigma} \times m_{\sigma}$,
with entries in the reals. 
This is in general not correct, the correct statement is 
the main theorem stated above.

Another consequence of the mistake in~\cite{mcmullenmonson03}
is that the 
$\Lambda$-orthogonal basis described in~\cite{mcmullen14}
is in general too small.
In Section~\ref{sec:orthog}, we briefly discuss the relation between
McMullen's $\Lambda$\nbd inner product and some other natural
inner products, and indicate how to 
construct a complete orthogonal basis.

In Section~\ref{sec:spherical}, we discuss some relations
between McMullen's 
cosine vectors and the spherical functions
appearing in the theory of Gelfand pairs.
It turns out that when $(G,H)$ is a Gelfand pair
(where $H$ is the stabilizer of a vertex), then
the cosine vectors are in principle the same as 
the spherical functions.
(This applies to all classical regular polytopes
in euclidean space, except the $120$-cell.)
We show that the values of cosine vectors are algebraic numbers, 
when the essential Wythoff dimension is $1$. 
This was conjectured by McMullen~\cite[Remark~9.4]{mcmullen14}.
Indeed, multiplied with the size of the corresponding layer,
we get an algebraic \emph{integer}.

Finally, in Section~\ref{sec:600cell} we propose an explanation
of an observation of McMullen~\cite[Remark~9.3]{mcmullen11} 
about the cosine vectors of
the $600$\nbd cell.

\section{Realizations as 
   \texorpdfstring{$G$-homomorphisms}{G-homomorphisms}}
Let $G$ be a finite group.
For convenience, we use the following terminology:
An \defemph{euclidian $G$-space} is an euclidean vector space
$(V,\langle \cdot, \cdot \rangle)$ 
on which the group~$G$ acts by orthogonal transformations.
The action is denoted by $(v,g)\mapsto vg$.
Equivalently, we are given an orthogonal representation
$D\colon G\to \Og(V)$, so that $D(g)$ is the map
$v\mapsto vg = vD(g)$.

Let $\Omega$ be a transitive (right) $G$\nbd set.
A \defemph{realization} of $(G,\Omega)$ 
is a map $A\colon \Omega \to V$ into 
an euclidean $G$\nbd space $V$ such that
$(\omega g)A = (\omega A) g$
for all $\omega\in \Omega$ and $g\in G$.
This definition agrees with
McMullen's definition~\cite{mcmullen89, mcmullen11, mcmullen14}
in the case where $G$ is the automorphism group 
of an abstract regular polytope
with vertex set $\Omega$.
We emphasize that in this paper,
$G$ is just some finite group and $\Omega$ a
transitive $G$\nbd set.
For example, we could take $\Omega=G$, on which $G$ acts by right
multiplication.

Two realizations 
$A_1\colon \Omega \to V_1$ and
$A_2\colon \Omega \to V_2$ 
are called \defemph{congruent}, if there is a linear isometry
$\sigma$ from the linear span of
$\{\omega A_1\mid \omega\in \Omega\}$
into $V_2$ such that
$A_1\sigma = A_2$.
(A peculiarity of this definition is that the realization
$\Omega \to \reals$ sending every $\omega\in \Omega$ to $0$
is \emph{not} congruent to the realization sending every
$\omega \in \Omega$ to $1$.
It turns out to be useful to distinguish these.)
The following is easy to see: 
\begin{lemma}
  Two realizations
  $A_1\colon \Omega \to (V_1, \langle \cdot, \cdot \rangle_1)$ and
  $A_2\colon \Omega \to (V_2, \langle \cdot, \cdot \rangle_2)$ 
  are congruent if and only if
  $\langle \xi A_1, \eta A_1\rangle_1
   = \langle \xi A_2, yA_2 \rangle_2$
  for all $\xi$, $\eta \in \Omega$.
\end{lemma}

Thus a realization $A\colon \Omega\to V$
is determined up to congruence by the $\Omega\times \Omega$ matrix
$Q = Q(A) $ with entries
$q_{\xi,\eta }= \langle \xi A, \eta A \rangle$.
We call $Q$ the \defemph{inner product matrix} of 
the realization~$A$.
It is 
a symmetric positive semi-definite matrix
and $G$\nbd invariant in the sense that
$q_{\xi g,\eta g}= q_{\xi,\eta }$.

\begin{remark}\label{rem:inner_prod_vecs}
McMullen~\cite{mcmullen11} uses 
\emph{inner product vectors} instead of inner product matrices.
The relation is as follows:
A \defemph{diagonal class} is an orbit of $G$ on the set of
unordered pairs on $\Omega$. 
Since the inner product matrix $Q=(q_{\xi,\eta})$ 
is symmetric and $G$\nbd invariant,
the map $\{\xi,\eta\}\mapsto q_{\xi,\eta}$
is well defined and constant along diagonal classes.
Thus it is determined by its values on a set
of representatives of the diagonal classes.

Now fix some ``initial'' vertex $\alpha\in \Omega$.
A \defemph{layer} is the set of all elements
$\omega\in \Omega$ such that $\{\alpha,\omega\}$
belongs to the same diagonal class.
Choose a set of representatives 
$\xi_0 = \alpha$, $\xi_1$, $\dotsc$, $\xi_r$
of the layers in $\Omega$.
Then the unordered pairs 
$\{\alpha,\xi_i\}$ represent all diagonal classes
(as $\Omega$ is a transitive $G$\nbd set).
The vector of length $r+1$ with values 
$q_{\alpha,\xi_i} = \langle \alpha A, \xi_i A \rangle $
as entries
is the \defemph{inner product vector} 
of the realization~\cite{mcmullen11}.
It is clear that the inner product matrix 
is determined by the inner product vector. 
For the purposes of this paper,
we find it more convenient to use 
the inner product matrix itself.
\end{remark}

The set of all inner product matrices of realizations of $\Omega$ 
is called the
\defemph{realization cone} of $\Omega$,
and denoted  by $\RC(\Omega)$ or $\RC(G,\Omega)$
(in the first variant,
the group $G$ is understood to be implicit in $\Omega$).
It is in bijection to the set of all congruence classes
of realizations.

The following operations on realizations show that 
the realization cone is indeed a cone:
First, if $A_1\colon \Omega\to V_1$ and
$A_2 \colon \Omega \to V_2$ are two realizations,
then their \defemph{blend}
is the realization
$A_1 \oplus A_2\colon \Omega \to V_1 \oplus V_2$
sending $\omega \in \Omega$ to
$(\omega A_1, \omega A_2)$ in the 
(outer) orthogonal  sum of the two euclidean spaces
$V_1$ and $V_2$.
(McMullen denotes the blend by
$A_1 \mathop{\#} A_2$.)
For the corresponding inner product matrices, we have
$Q(A_1 \oplus A_2) = Q(A_1)+ Q(A_2)$.

Second, we can scale realizations:
for $A\colon \Omega\to V$ and $\lambda\in \reals$,
$\lambda A\colon \Omega \to V$ is defined by
$\omega (\lambda A)= \lambda (\omega A)$.
Obviously, $Q(\lambda A) = \lambda^2 Q(A)$.

For completeness, we mention a third operation, 
the \emph{tensor product} 
$A_1\otimes A_2\colon \Omega \to V_1 \otimes V_2$
of two realizations $A_i \colon \Omega\to V_i$,
defined on $\Omega$ by
$\omega (A_1 \otimes A_2) 
 := (\omega A_1) \otimes (\omega A_2)$.
The inner product matrix 
$Q(A_1\otimes A_2)$ is the entry-wise (Hadamard) product
of $Q(A_1)$ and $Q(A_2)$.

It follows from blending and scaling that $\RC(\Omega)$ is a 
convex cone.
It is also clear that $\RC(\Omega)$ has an apex at $0$.

A realization $A\colon \Omega\to V$ is called
\defemph{normalized}, if 
$\norm{\omega A}^2 := \scp{\omega A, \omega A} = 1$ 
for some (and hence for all)
$\omega \in \Omega$.
If $\omega A\neq 0$, then we may scale
the realization by $1/\norm{\omega A}$,
so that it becomes normalized.
The inner product matrix of the normalized realization
$(1/\norm{\omega A}) A$
is called its \defemph{cosine matrix}, for obvious reasons.
The set of cosine matrices of realizations forms 
a compact convex set. 
\begin{remark} 
As in Remark~\ref{rem:inner_prod_vecs}, a cosine matrix corresponds
to a \defemph{cosine vector}, which contains the values
$\scp{\alpha A, \xi_i A}/\scp{\alpha A, \alpha A}$, where $\xi_i$ 
runs over a set of representatives of the layers.
We have to caution the reader that McMullen~\cite{mcmullen14}
uses the term \emph{cosine matrix} with a different meaning:
In~\cite{mcmullen14}, this is a square matrix whose rows are cosine
vectors of different realizations
(and maybe certain \emph{mixed cosine vectors}), and such that
the rows are orthogonal with respect to a certain inner product
($\Lambda$-orthogonality, see Section~\ref{sec:orthog} below).
This matrix is similar to the character table of a finite group,
and thus we find the name ``cosine table'' more appropriate
for this object.
\end{remark}

An especially important realization is the
\defemph{simplex realization} which we now define.
Recall that the permutation module $\reals \Omega$ 
over $\reals$ belonging to the $G$\nbd set  $\Omega$ 
is the set of formal sums
\[ \reals\Omega :=
   \{\sum_{\omega\in \Omega} r_{\omega} \omega
           \mid r_{\omega}\in \reals 
   \}, 
\]
on which $G$ acts by 
$(\sum r_{\omega}\omega)g = \sum r_{\omega}(\omega g)$.
Also we think of $\reals \Omega$ as equipped with the 
standard scalar product
\[ \scp*{\sum_{\omega} r_{\omega} \omega, 
        \sum_{\omega} s_{\omega} \omega
       }
   = \sum_{\omega} r_{\omega}s_{\omega}.
\]
This makes $\reals \Omega$ into an euclidean $G$\nbd space.
The natural map $\Omega \into \reals \Omega$
is a realization, called the
\emph{simplex realization}.
(We usually identify its image, the canonical basis of
$\reals\Omega$, with the set $\Omega$.)

The next observation is obvious, but crucial for our 
proof of the structure theorems in the next section.
Recall that a linear map
 $\widehat{A}\colon U\to V$ between two $ G$\nbd modules is a
 $G$\nbd module homomorphism if
  $ug\widehat{A}= u\widehat{A} g$ for all
  $u\in U$ and $g\in G$.
Since $\Omega$ is a basis of $\reals \Omega$,
we have the following:
\begin{obs}
  Realizations $A\colon \Omega \to V$ correspond to
  $G$\nbd module homomorphisms 
  $\widehat{A}\colon \reals \Omega \to V$.
\end{obs}
From now on, we identify a realization $A\colon \Omega\to V$
with the corresponding linear map $\reals \Omega\to V$,
and use the same letter $A$ for both.
We also identify 
a linear map $A\colon \reals \Omega\to V$
with its matrix $A$ with respect to the canonical 
basis $\Omega$
and some fixed orthonormal basis of $V$.
The inner product matrix of the realization $A$
is then $Q= AA^t$, and does not depend on the choice of basis
of $V$.

We also write $A^t\colon V\to \reals \Omega$
for the \emph{adjoint map} of 
$A\colon \reals \Omega\to V$ with respect to the
inner products on $\reals \Omega$ and $V$;
if $A$ is a $G$\nbd module homomorphism, then so is $A^t$.
From this viewpoint, $Q=AA^t$ is a $G$\nbd module endomorphism
of $\reals \Omega$.
\begin{thm}\label{t:ipms_endos}
  Let $\Omega$ be a transitive $G$\nbd set.
  Then
  \begin{align*}
     \RC(\Omega)
        &= \{AA^t \mid A\in \mat_{\Omega}(\reals)
             \text{ is $G$-invariant} \},
  \end{align*}
  and this equals the set of $G$-invariant, symmetric positive
  semi-definite matrices.
\end{thm}
This is the special case $U=\reals \Omega$ of the
following general observation:
\begin{lemma}\label{l:commuting_semidef}
  Let $U$ be an euclidean $G$\nbd space 
  and $Q\in \enmo_{\reals}(U)$.
  The following are equivalent:
  \begin{enumthm}
  \item \label{it:pseudosquare}
        There is an euclidean $G$\nbd space $V$
        and a $G$\nbd homomorphism $A\colon U\to V$
        such that $Q = AA^t$.        
  \item \label{it:sps_inv}
        $Q$ is symmetric positive semi-definite and 
        commutes with $G$.
  \item \label{it:square}
        There is  
        $A\in \enmo_{\reals G}(U)$ such that
        $Q= AA^t$.
  \end{enumthm}  
\end{lemma}
\begin{proof}
  Obviously, \ref{it:square} is a special case 
  of~\ref{it:pseudosquare},
  and~\ref{it:pseudosquare} implies~\ref{it:sps_inv}.
  
  It remains to show that~\ref{it:sps_inv} implies~\ref{it:square},
  so assume $Q$ is symmetric positive semi-definite and
  commutes with $G$.
  Then $U$ is the orthogonal sum of the eigenspaces
  of $Q$, and the eigenvalues of $Q$ are non-negative real numbers.
  For each eigenvalue $\lambda$ of $Q$, let
  $p_{\lambda}\colon U \to U$
  be the orthogonal projection onto the corresponding eigenspace
  of $Q$.
  Since 
    $Q$ commutes with $G$, it follows that the eigenspaces 
    are $G$-invariant and thus
  the $p_{\lambda}$'s commute with $G$.
  
  Since $U$ is the orthogonal sum of the eigenspaces,
  we have $\id_U = \sum_{\lambda} p_{\lambda}$.
  For $u\in U$, it follows
  \[ uQ = \sum_{\lambda} up_{\lambda}Q 
        = \sum_{\lambda} \lambda (up_{\lambda})
        = u \sum_{\lambda} \lambda p_{\lambda}.
  \]
  Since 
  $p_{\lambda} p_{\mu}
    = \delta_{\lambda,\mu} p_{\lambda}$
  for eigenvalues $\lambda$, $\mu$ of $Q$,
  and since all $\lambda \geq 0$, we get
  \[ Q = \sum_{\lambda} \lambda p_{\lambda}
       = \left( \sum_{\lambda}  \sqrt{\lambda} 
                                  p_{\lambda}
         \right)^2.
  \]
  Set
  $A= \sum_{\lambda} \sqrt{\lambda} p_{\lambda}$,
  an element commuting with $G$.
  Then $A = A^t$, 
  since orthogonal projections are self-adjoint, 
  and thus $Q = A^2= AA^t $ as required.  
\end{proof}

\section{The structure of the realization cone}
\label{sec:structure}
In this section, we determine the structure of the realization cone.
The general idea is the following: 
We can write the module $\reals \Omega$ as an orthogonal sum 
of simple modules, say
\[ \reals \Omega \iso m_1 S_1 \oplus \dotsb 
                           \oplus m_k S_k,
\]
with natural numbers $m_i$, 
and where the different $S_i$'s are non-isomorphic.
It is well known that then
\[ \enmo_{\reals G}(\reals\Omega)
      \iso \mat_{m_1}(\enmo_{\reals G}(S_1))
           \times \dotsm \times
           \mat_{m_k}(\enmo_{\reals G}(S_k)),
\]
where for each $i$ the endomorphism ring
$\mathbb{D}_i := \enmo_{\reals G}(S_i)$
is a division ring by Schur's lemma, and thus
either $\reals$, $\compl$ or $\quats$.
The aim of this section is to fill in the details and to show that
the above isomorphism, when restricted 
to the realization cone 
$\RC(\Omega)$ as a subset of $\enmo_{\reals G}(\reals\Omega)$,
yields a similar decomposition into subcones of the 
form $\{AA^{*}\mid A \in \mat_{m_i}(\enmo_{\reals G}(S_i))\}$.

We begin by recalling some
general representation theory.
As usual, we write $\Irr G$ for the set of irreducible complex
characters of a group $G$.
Furthermore, $\Irr_{\reals}G$ denotes the set of characters of
simple $\reals G$\nbd modules
(equivalently, of irreducible representations 
 $G\to \GL(d,\reals)$).
For class functions $\alpha$, $\beta\colon G\to \compl$,
 \[ \ipcf{\alpha,\beta} := \frac{1}{\abs{G}}
                 \sum_{g\in G} \alpha(g)\overline{\beta(g)}
 \]
denotes the usual inner product of class functions.
It is well known that $\Irr G$ is an orthonormal basis
of the space of class functions
with respect to this inner product.
For $\sigma\in \Irr_{\reals} G$, we have the following
possibilities~\cite[III.5A]{simonreps}\cite[Ch.~4]{isaCTdov}:
\begin{lemma}\label{l:realreps}
  Let $S$ be a simple $\reals G$\nbd module with
  character $\sigma\in \Irr_{\reals} G$.
  Then one of the following three cases occurs:
  \begin{enumthm}
  \item $\ipcf{\sigma,\sigma}=1$, 
        $\sigma \in \Irr G$
        and\/ $\enmo_{\reals G}(S) \iso \reals$,
  \item $\ipcf{\sigma,\sigma} = 2$, 
        $\sigma = \chi + \overline{\chi}$
        with $\chi\neq \overline{\chi}\in \Irr G$
         and\/
        $\enmo_{\reals G}(S)\iso \compl$, 
  \item  $\ipcf{\sigma,\sigma} = 4$, 
         $\sigma = 2\chi $
          with $\chi = \overline{\chi}\in \Irr G$
          and\/
          $\enmo_{\reals G}(S)\iso \quats$.     
  \end{enumthm}
\end{lemma}
We call $S$ and $\sigma$ of real, complex or quaternion type, 
respectively.

Let $S$ be a simple $\reals G$-module with character $\sigma$. 
For any  $\reals G$\nbd module $V$,
  let $V_{\sigma}= V_S$ be the sum of all submodules 
  of $V$ isomorphic to $S$.
  The submodule $V_{\sigma}$ is called the 
  $\sigma$\nbd homogeneous component of $V$.
Every module $V$ is the direct sum of the $V_{\sigma}$,
  as $\sigma$ runs over $\Irr_{\reals}G $.
This sum is orthogonal with respect to any 
$G$\nbd invariant inner product defined on $V$.
The orthogonal projection $V\to V_{\sigma}$ is  given by
the action of
\[  e_{\sigma}
    = \frac{ \sigma(1) }{ \ipcf{\sigma,\sigma} \abs{G} }
      \sum_{g\in G} \sigma(g^{-1}) g \in \Z(\reals G)
\]
on $V$.
(The formula for the idempotent $e_{\sigma}$
follows from the analogous one in the complex
case \cite[Theorem~2.12]{isaCTdov}\cite[III.7]{simonreps}
together with Lemma~\ref{l:realreps}.)
We have
\[ 1 = \sum_{\sigma\in \Irr_{\reals} G} e_{\sigma},
   \quad \text{and} \quad
   e_{\sigma}e_{\tau} = \delta_{\sigma,\tau} e_{\sigma}
   \quad \text{for all}\: \sigma, \tau \in \Irr_{\reals}G.
\]
Notice that since $e_{\sigma} \in \Z(\reals G)$,
the action of $e_{\sigma}$ on modules commutes with both the 
action of $G$ and the action of $G$\nbd module homomorphisms.

For each $\sigma\in \Irr_{\reals}G$,
define $\RC_{\sigma}(\Omega)$ to be the set of all 
inner product matrices which arise from a realization
$A\colon \Omega \to V$ such that
$V = V_{\sigma}$, 
so $V$ has character $k\sigma$ for some $k\in \nats$.
Equivalently, if $S$ is an irreducible module
affording $\sigma$, then $V$ is isomorphic to a direct sum
of copies of $S$.
(The subcone $\RC_{\sigma}(\Omega)$ is denoted by $\mathcal{P}_D$ 
 in~\cite{mcmullen89,mcmullenmonson03},
 where $D$ is an irreducible representation of $G$
 affording $\sigma$.)
 
In the next result, we view both the inner product matrix
and the idempotent $e_{\sigma}$ as operators on the 
permutation module $\reals \Omega$.
\begin{thm}\label{t:subconeidemdec}
  (cf. \cite[Theorem~16]{mcmullen89}, 
   \cite[Theorem~4.1]{mcmullenmonson03})
  $\RC_{\sigma}(\Omega)$ is a closed subcone of $\RC(\Omega)$
  and $\RC(\Omega)$ is the direct sum of the 
  $\RC_{\sigma}(\Omega)$,
  where $\sigma\in \Irr_{\reals}G$.
  More precisely, for $Q\in \RC(\Omega)$, we have
  \[ Q = \sum_{\sigma\in \Irr_{\reals} G} Q_{\sigma}
     ,\quad\text{where}\quad
     Q_{\sigma} = e_{\sigma}Q = Qe_{\sigma}
     \in \RC_{\sigma}(\Omega).
  \]
  (In particular, $Q\in \RC_{\sigma}(\Omega)$ if and only if
    $e_{\sigma}Q=Q$, if and only if $Q=Qe_{\sigma}$.)  
\end{thm}
This means that if the inner product matrix $Q$ of a realization has
  entries $q_{\xi, \eta}$, then 
  the inner product matrix 
  $Q_{\sigma} = e_{\sigma}Q$ of the 
  $\sigma$\nbd homogeneous component of the realization has
  entries 
  \[  s_{\xi,\eta}
      := \frac{\sigma(1)}{\ipcf{\sigma,\sigma}\abs{G}}
                     \sum_{g\in G}\sigma(g^{-1})
                                  q_{\xi g, \eta}
      \quad \text{for all $\xi$, $\eta\in \Omega$}. 
  \]
\begin{proof}[Proof of Theorem~\ref{t:subconeidemdec}]
  Suppose $A\colon \reals \Omega \to V$ is a realization
  with inner product matrix 
  $Q=AA^t\in \RC(\Omega)$.
  Then $e_{\sigma}A = Ae_{\sigma}$ is a realization
  $\reals \Omega \to Ve_{\sigma}$ with inner product matrix
  $(e_{\sigma}A)(e_{\sigma}A)^t 
    = e_{\sigma} Q e_{\sigma}
    = e_{\sigma} Q$,
  since $e_{\sigma}^t = e_{\sigma}= e_{\sigma}^2$.
  Thus $e_{\sigma}Q$ is an inner product matrix
  in $\RC_{\sigma}(\Omega)$.
  Conversely, if $Q\in \RC_{\sigma}(\Omega)$,
  then $Q=AA^t$ for some realization $A$ with
  $A= Ae_{\sigma} $, and thus
  $Q=e_{\sigma}Q$.
  
  Since $Q = \sum_{\sigma} e_{\sigma}Q$
  for any inner product matrix, the result follows.
\end{proof}
(That $\RC_{\sigma}(\Omega)$ is a subcone and that 
 $\RC(\Omega)$ is the sum of these subcones is also immediate
 from the equation 
 $Q(A_1\oplus A_2)=Q(A_1)+Q(A_2)$
 and the fact that every $\reals G$\nbd module can be written
 as an orthogonal sum of simple modules.)

Next we determine the structure of 
$\RC_{\sigma}(\Omega)$, for $\sigma\in \Irr_{\reals}G$.
Let $S$ be a simple $\reals G$\nbd module affording 
$\sigma$.
We can write 
$(\reals \Omega)_{\sigma}$ as the orthogonal sum of 
$m=m_{\sigma}=m_S$ copies of
$S$, that is, $(\reals \Omega)_{\sigma} \iso m S$.
The non-negative integer $m$ is called
the \emph{multiplicity} of $S$ in $\reals \Omega$
and of $\sigma$ in the character $\pi = (1_H)^G$
of $\reals \Omega$.
In other words, we have
\[ \pi = (1_H)^G = \sum_{ \sigma\in \Irr_{\reals}G }
                m_{\sigma} \sigma, \]
and this equation determines the $m_{\sigma}$'s.
(Here $H=G_{\alpha}$, the stabilizer of a vertex $\alpha$.) 

Recall that the \emph{Wythoff space} $W_S$
associated to $S$ (and $\alpha\in \Omega$) is the fixed space
of $H$ on $S$.
McMullen and Monson~\cite{mcmullenmonson03}
defined the \emph{essential Wythoff dimension}
as the dimension of $W_S$ over the centralizer ring
$\mathbb{D} = \enmo_{\reals G}(S)$.
\begin{lemma}
  The multiplicity $m_S = m_{\sigma}$ 
  equals the essential Wythoff dimension.
\end{lemma}
\begin{proof}
  Let $\pi$ be the character of $\reals \Omega$.
  Then 
  $\ipcf{\pi, \sigma}_G = m_{\sigma} \ipcf{\sigma,\sigma}_G 
                        = m_{\sigma} \dim_{\reals}(\mathbb{D}) $.
  On the other hand, $\pi = (1_H)^G$ and
  $\ipcf{\pi,\sigma}_G = \ipcf{1_H,\sigma_H}_H
   = \dim_{\reals} W_S$ by Frobenius reciprocity.
  The result follows. 
\end{proof}

Before we give our structure theorem for $\RC_{\sigma}(\Omega)$,
we digress to reprove Theorems~4.4 and~4.5 of
the McMullen-Monson paper~\cite{mcmullenmonson03}, 
since, as we argue below, McMullen's and Monson's proofs 
of these theorems are not correct.

We recall that a realization $A\colon\reals \Omega \to V$ 
and the corresponding polytope
are called \defemph{pure}, 
when the image $A(\reals \Omega)$ is simple as module over $G$.
The following contains Theorems~4.4 and~4.5 from
the paper of McMullen and Monson~\cite{mcmullenmonson03}.
\begin{thm}
  Every polytope in $\RC_{\sigma}(\Omega)$ is the blend of
  at most $m_{\sigma} $ pure polytopes,
  and has dimension at most $m_{\sigma} \sigma(1)$,
  where $m_{\sigma}\sigma(1)$ is possible.
\end{thm} 
\begin{proof}
  Let $A\colon \Omega \to V$ be a realization, 
  which we identify as usual with a
  $G$\nbd homomorphism $\reals \Omega \to V$.
  Without loss of generality, we can assume that
  $V=(\reals \Omega)A$, that is, 
  $V$ is the linear span of
  $\{\omega A \mid \omega\in \Omega\}$.
  The orthogonal complement of 
  $\Ker A$ in $\reals \Omega$ is a $G$\nbd invariant subspace 
  isomorphic to $V$.
  In particular, if $V \iso k S$, where $S$ affords $\sigma$,
  it follows from the uniqueness of the decomposition
  of $\reals \Omega$ into irreducible summands that
  $k \leq m_{\sigma}$.
  Then $A$ is the blend of $k$ pure realizations,
  and the polytope spanned by 
  $\{\omega A \mid \omega\in \Omega\} $
  has dimension $k\sigma(1) \leq m_{\sigma}\sigma(1)$.
  Finally, $e_{\sigma}$ viewed as realization
  $\reals \Omega \to U = \reals \Omega e_{\sigma}$
  yields a polytope of dimension 
  $\dim U = m_{\sigma}\sigma(1)$.  
\end{proof}

In the description of $\RC_{\sigma}(\Omega)$,
we use the following notation:
for a matrix $B$ over the complex numbers or the quaternions,
$B^{*}$ denotes the transposed conjugate.
If $B$ has real entries, then $B^{*}=B^t$, 
the transposed matrix.
\begin{thm}\label{t:homog_subconedim}
  Let $S$ be a simple module affording 
  $\sigma\in \Irr_{\reals}G$,
  let $m=m_{\sigma}$ be its multiplicity in $\reals \Omega$ and 
  set\/ $\mathbb{D} = \enmo_{\reals G}(S)$.
  Then
  \[ \RC_{\sigma}(\Omega)
      \iso \{ BB^{*}\mid B\in \mat_{m}(\mathbb{D})
           \}.
  \]
\end{thm}
\begin{example}\label{expl:120cell}
  Let $\Omega$ be the vertex set of the $120$\nbd cell
  (of size $600$) and $G$ its 
  symmetry group.
  Using the computer algebra system GAP~\cite{GAP476},
  one can compute the multiplicities of the irreducible characters
  in the permutation character.
  There are $15$ characters occurring with multiplicity $1$,
  three characters occurring with multiplicity $2$
  (of degrees $16$, $16$ and $48$),
  and two characters occurring with multiplicity $3$
  (of degrees $25$ and $36$).
  All characters are of real type.
  The realization cone of the $120$\nbd cell 
  is thus a direct product
  of $15$ copies of $\reals_{\geq 0}$,
  of three copies of the cone of symmetric positive semidefinite 
  $2\times 2$-matrices, 
  and two copies of the cone of symmetric positive
  semidefinite $3\times 3$-matrices.
  The $120$\nbd cell is the only classical regular polytope for 
  which the realization cone is not polyhedral.  
\end{example}
A corollary of the theorem is the correct version 
of~\cite[Theorem~4.6]{mcmullenmonson03}.
\begin{cor}\label{c:dimsubcone}
  We have
  \begin{align*} 
    \dim \RC_{\sigma}(\Omega)
      &= m + \frac{m(m-1)}{2}\ipcf{\sigma,\sigma}
      \\
      &=
     \begin{cases}
        \frac{m(m+1)}{2} &\text{for}\quad \mathbb{D}\iso \reals,\\
        m^2 &\text{for}\quad \mathbb{D}\iso \compl,\\
        m(2m-1)  &\text{for} \quad \mathbb{D}\iso \quats.
     \end{cases}
  \end{align*}
\end{cor}
\begin{proof}
  It follows from Theorem~\ref{t:homog_subconedim}
  that the linear span of $\RC_{\sigma}(\Omega)$
  is isomorphic to the $m\times m$
  self-adjoint matrices over $\mathbb{D}$.
  Since $\ipcf{\sigma,\sigma}= \dim_{\reals}(\mathbb{D})$, 
  the result follows.
\end{proof}
In the proof of Theorem~\ref{t:homog_subconedim}, and also later,
we need the following simple observation:
\begin{lemma}\label{l:adjointconj}
  Let $S$ be an irreducible euclidean $G$\nbd space
  and let\/ $\mathbb{D}=\enmo_{\reals G}(S)$.
  Then for $d\in \mathbb{D}$ we have
  $d^{t}= \overline{d}$
  (that is, the adjoint map with respect to the scalar product
  on $S$ equals the complex/quaternion conjugate).
\end{lemma}
\begin{proof}
  We have $d^t \in \mathbb{D}$ again and thus $dd^t\in \mathbb{D}$.
    The eigenspaces of $dd^t$ on $S$ are $G$-invariant, and thus
    $dd^t = \lambda \id_S$ with $\lambda\in \reals_{\geq 0}$.
    This means that 
    $\langle vd,vd \rangle =\lambda\langle v,v \rangle$ for all
    $v\in S$.
    For $d=i$ (or $d\in \{i,j,k\}$ when $D=\quats$),
    it follows $\lambda=1$ 
    (because $\lambda^2\langle v, v\rangle 
             = \langle vd^2, vd^2 \rangle 
             = \langle -v, -v \rangle$), 
    and thus
    $d^t=\overline{d}$ in this case.
    The general case follows from this.
\end{proof}
\begin{proof}[Proof of Theorem~\ref{t:homog_subconedim}]
  First, observe that it follows from
  Theorem~\ref{t:ipms_endos} together with
  Theorem~\ref{t:subconeidemdec} that
  \[\RC_{\sigma}(\Omega)
    = \{ AA^t \mid A\in \enmo_{\reals G}(\reals\Omega), 
                   Ae_{\sigma}= A\}.
  \]
  Fix a $G$\nbd invariant inner product
  $\langle \cdot ,\cdot\rangle_S$ on the simple
  module $S$ affording $\sigma$.
  Suppose that $\mu\colon S\to \reals \Omega$
  is an isomorphism from $S$ onto some simple
  submodule of $\reals \Omega$
  (necessarily, $S\mu\subseteq \reals \Omega e_{\sigma}$).
  After eventually scaling $\mu$, we may assume that
  $ \langle v, w \rangle_S 
     = \langle v\mu, w\mu \rangle_{\reals \Omega} $.
  Then with $\pi = \mu^t\colon \reals \Omega\to S$, we have
  $\mu\pi = \id_S$ and $\pi\mu$ is the orthogonal projection
  from $\reals \Omega$ onto $S\mu$.
  We know that $\reals \Omega e_{\sigma}$ is isomorphic to 
  a sum of $m$ copies of $S$.
  Thus we can find $G$-module homomorphisms  
  $\mu_i\colon S\to \reals \Omega$ and 
  $\pi_i\colon \reals\Omega\to S$, $i=1$, $\dotsc$, $m$, 
  such that
  \[ \pi_i = \mu_i^t, \quad
     \mu_i\pi_j = \delta_{ij}\id_S 
     , \quad \text{and}\quad
     e_{\sigma} = \sum_{i=1}^m \pi_i\mu_i.
  \]
  Using these maps, we can describe the algebra isomorphism between
  \[ \{A\in \enmo_{\reals G}(\reals \Omega)\mid Ae_{\sigma}= A
     \} 
     \quad\text{and}\quad 
     \mat_m(\mathbb{D}),
  \]
  where $\mathbb{D}=\enmo_{\reals G}(S)$:
  Send $A\in \enmo_{\reals G}(\reals \Omega)$ to the 
  matrix $(\mu_i A \pi_j)\in \mat_m(\mathbb{D})$.
  Conversely,
  map a matrix $(b_{ij})$ to $\sum_{i,j} \pi_i b_{ij}\mu_j$.

  This isomorphism sends the adjoint map 
  $A^t$ to the matrix
  $(\mu_i A^t \pi_j)= ( \pi_i^t A^t\mu_j^t)
    = ( (\mu_j A \pi_i)^t) = (\overline{\mu_j A \pi_i})$,
  where the last equality follows from Lemma~\ref{l:adjointconj}.
  Thus it sends a inner product matrix 
  $AA^t$ to a matrix $BB^{*}$ as claimed.
\end{proof}

Finally, Theorem~4.7(b) of 
McMullen and Monson~\cite{mcmullenmonson03} 
has to be modified accordingly.
\begin{cor}
  Let $r+1$ be the number of layers. Then
  \begin{align*}
    r+1 &= 
     \sum_{\sigma\in \Irr_{\reals} G}
       m_{\sigma}
     + \sum_{\sigma\in \Irr_{\reals} G}
        \frac{ m_{\sigma} (m_{\sigma}-1) }{2} \ipcf{\sigma,\sigma}.
  \end{align*}
\end{cor}
We can rewrite the right hand side of the above formula 
in terms of the irreducible complex characters.
Recall that 
$ m_{\sigma} = \ipcf{(1_H)^G,\sigma}/\ipcf{\sigma,\sigma}$.
Thus if $\sigma =\chi\in \Irr G$ 
or $\sigma= \chi +\overline{\chi}$
with $\chi\neq \overline{\chi}$,
then $m_{\sigma}=m_{\chi}$ ($=\ipcf{(1_H)^G,\chi}$),
and if $\sigma = 2\chi$ with 
$\chi=\overline{\chi}\in \Irr G$,
then $m_{\sigma}=m_{\chi}/2$. 
Also recall the Frobenius-Schur indicator 
$\nu_2(\chi) =(1/\abs{G})\sum_g \chi(g^2)$,
which is $1$, $0$ and $-1$, respectively, 
in the three mentioned cases.
Using all this, one can derive the following equation:
\[ r+1 = \frac{1}{2} 
         \sum_{\chi\in \Irr G} 
             m_{\chi}(m_{\chi}+\nu_2(\chi)).
\]
Herman and Monson~\cite{hermanmonson04}
derived this equation from Frame's formula for the number of 
symmetric cosets.
Conversely, we can derive Frame's formula from the last equation.

We conclude this section with a discussion about what is actually 
wrong in McMullen's and Monson's proof~\cite{mcmullenmonson03}.
The mistake is that the \emph{essential Wythoff space}
defined before Theorem~4.4 has not all the properties
the authors assume (implicitly).
It is in general not true that a traverse of the action of the 
unit complex numbers (or the unit quaternions)
can be chosen as a subspace.
For example, if the Wythoff space $W$ has dimension $4$ 
over the reals and if the centralizer ring is the field 
$\compl$ of complex numbers,
then $W\iso \compl^2$.
Clearly, not every element of $\compl^2$ can be written as
$v\cdot z$ with $v\in \reals^2$, $z\in \compl$ and $\abs{z}=1$,
for example, $(1,i)$ is not of this form.
On the other hand, in the $\reals$\nbd linear hull of
$\reals^2 \cup \{(1,i)\}$ we have the vector
$ -(1,0) + (1,i)= (0,i) = (0,1)i$, so this is no longer a traverse
for the unit complex numbers.

Of course, we can always choose a $\mathbb{D}$\nbd basis 
of $W$ and then let $W^{*}$ be the $\reals$\nbd linear hull
of this basis.
This is what is essentially done in the proof of
Theorem~4.4 in~\cite{mcmullenmonson03}.
But then the sentence ``The general pure polytope in 
$\mathcal{P}_G$
arises from a point
$\alpha_1p_1 + \dotsb + \alpha_{w^*}p_{w^*}\in W^*$''
is no longer true. 
We should allow coefficients $\alpha_i\in \mathbb{D}$,
but then different points in the Wythoff space yield congruent
realizations.
So the proof must be modified somehow.

This flaw in the arguments also bears upon results in the later 
paper~\cite{mcmullen14}.
Namely, in Theorem~5.2 there and the remarks before, 
the definition of the matrix $A$ has to be modified,
allowing for entries in the centralizer ring.
We may view Theorem~\ref{t:homog_subconedim} above as the correct
version of~\cite[Theorem~5.2]{mcmullen14}.
The $\Lambda$\nbd orthogonal basis described in Sections~5 and 6
of~\cite{mcmullen14}
does not generate the full space of cosine vectors,
if there is $\sigma$ with $m_{\sigma}>1$ and
$\mathbb{D}_{\sigma} \not\iso \reals$,
and has to be modified accordingly.
(We will consider this below in Section~\ref{sec:orthog}.) 

\section[Counterexamples]{Counterexamples to a result of %
Herman and Monson}
\label{sec:counterex}
The main case of interest of the preceding theory is when 
$\Omega$ is the vertex set of an abstract regular polytope
$P$ and $G$ is the automorphism group of $P$.
Equivalently, $G= \erz{s_0, s_1, \dotsc, s_{n-1}}$
is a \emph{string C-group}
and $H=\erz{s_1,\dotsc,s_{n-1}}$ is the stabilizer of some element
of $\Omega$.
By definition, 
this means that the generators $s_0$, $s_1$, $\dotsc$ are
involutions, that the 
\emph{intersection property}
\[ \erz{s_i \mid i \in I} \cap \erz{s_j \mid j\in J}
   = \erz{s_k\mid k\in I\cap J}
\]
holds for all subsets $I$, $J\subseteq \{0,1,\dotsc,n-1\}$,
and that $s_is_j = s_js_i$ for
$\abs{i-j}\geq 2$.
Since the polytope can be recovered
from the group $G$ and the 
distinguished generators 
$s_0$, $s_1$, $\dotsc$, $s_{n-1}$
\cite[Section~2E]{mcmullenschulte02_arp},
we do not need to recall here what an abstract regular polytope
actually \emph{is}.
The concepts of abstract regular polytopes and string C-groups are,
in a certain sense,
equivalent, and we work solely with the latter.

We now give an example which shows that we can have
$m_{\sigma}>1$ for $\sigma$ of complex type,
even when $\Omega$ is the vertex set of an abstract regular 
polytope.
This shows that Theorem~2 in~\cite{hermanmonson04}
is wrong.
The example is a special case of a more general construction which
we will consider afterwards.
\begin{example}\label{expl:pslconcrete}
  Consider the matrices
  \[ S_0=
     \begin{pmatrix}
        0 & 1 \\
        -1 & 0
    \end{pmatrix}
    ,\quad
    S_1 = 
    \begin{pmatrix}
      0 & 2 \\
      9 & 0
    \end{pmatrix}
    ,\quad
    S_2 =
    \begin{pmatrix}
      8 & -7 \\
      -7 & -8
    \end{pmatrix}
    \in \SL(2,19).
  \]
  It is not difficult to see that their images
  $s_0$, $s_1$ and $s_2$ in $G:=\PSL(2,19)$
  generate $G$ and that $G$ is a string C-group with
  respect to these involutions
  (see Lemma~\ref{l:matrices_stringc} below).
  The element $s_1s_2$ has order $3$
  and thus $H=\erz{s_1,s_2} \iso S_3$ has order $6$.
  Now $G$ has an irreducible character $\chi$
  of degree $9$ with $\chi\neq \overline{\chi}$.
  We have $\ipcf{(1_H)^G,\chi}_G = \ipcf{1_H,\chi_H}=2>1$.
  Thus the corresponding irreducible module over the reals
  has a Wythoff space of dimension $4$ and 
  essential Wythoff dimension (=multiplicity) $2$.
  (The corresponding abstract regular polytope
   has Schläfli type $\{9,3\}$.)
\end{example}
We are now going to show that there are in fact string C-groups with
irreducible representations
of complex type and arbitrary large essential Wythoff dimension.
The following is probably well known:
\begin{lemma}\label{l:matrices_stringc}
  Let $\mathbb{F}$ be a field.
  Let
  \[ S_0=
      \begin{pmatrix}
          0 & 1 \\
          -1 & 0
      \end{pmatrix}
      ,\quad
      S_1 = 
      \begin{pmatrix}
        0 & y \\
        -y^{-1} & 0
      \end{pmatrix}
      ,\quad
      S_2 =
      \begin{pmatrix}
        a & b \\
        b & -a
      \end{pmatrix}
      \in \SL(2,\mathbb{F}), \]
  where $y\neq 0, \pm 1$, $a^2+b^2=-1$ and $a\neq 0$.
  Then
  \[ G = \erz{ S_0, S_1, S_2 }/ \{\pm 1\} \leq \PSL(2,\mathbb{F})
  \]
  is a string C-group.
\end{lemma}
\begin{proof}
  Let $s_i$ be the image of $S_i$ in $\PSL(2,\mathbb{F})$.
  It is easily checked that $s_0$, $s_1$ and $s_2$
  are mutually distinct involutions and that $s_0s_2=s_2s_0$.
  
  It remains to check the intersection property.
  For this, it suffices to show that
  \[ \erz{s_0,s_1}\cap \erz{s_1,s_2} = \erz{s_1}=\{1,s_1\},
  \]
  the other equalities then 
  follow~\cite[Proposition~2E16]{mcmullenschulte02_arp}.
  We have 
  \[ \erz{s_0,s_1}\cap \erz{s_1,s_2}
     = \erz{s_1}C
     \quad \text{where} \quad
     C = \erz{s_0s_1}\cap \erz{s_1s_2},
  \]
  and we want to show that $C = \{1\}$.
  As
  \[ S_0S_1 = \begin{pmatrix}
                -y^{-1} & 0 \\ 0 & -y
              \end{pmatrix}
     , \quad y\neq y^{-1},
  \]
  the matrix $S_0S_1$ and its powers have 
  eigenvectors $(1,0)$ and $(0,1)$.
  Since
  \[ S_1S_2 = 
     \begin{pmatrix}
        yb & -ya \\ 
       -y^{-1}a & -y^{-1}b 
     \end{pmatrix}
     , \quad ya \neq 0,
  \]
  the vectors $(1,0)$ and $(0,1)$ are not eigenvectors of $S_1S_2$,
  but $S_1S_2$ has an eigenvector,
  possibly over an algebraic extension $\mathbb{E}$ of 
  $\mathbb{F}$.
  Thus the elements of $C$ fix three different lines in 
  $\mathbb{E}^2$, 
  and thus come from scalar matrices as claimed.
\end{proof}
The matrices in the last lemma have been used by 
Cherkassoff and Sjerve~\cite{cherkassoffsjerve94} 
to generate $\PSL(2,q)$ for $q\equiv -1\mod 4$, $q\geq 19$.
In fact, their argument shows the following, 
which is sufficient for our purposes:
\begin{lemma}\label{l:pslgenerate}
  In Lemma~\ref{l:matrices_stringc},
  let $\mathbb{F}$ be a field with $p$ elements,
  where $p$ is a prime and
  $p\equiv -1 \mod 4$,
  and let $s_i$ be the image of $S_i$ in $\PSL(2,p)$.
  If the order of $s_0s_1$ or $s_1s_2$ 
  is $\geq 6$, then $\erz{s_0,s_1,s_2}=\PSL(2,p)$.
\end{lemma}
\begin{proof}
  We use Dickson's classification of the subgroups 
  of $\PSL(2,p)$~\cite[Chapter~3, Theorem~6.25]{suzuki82GT1}.
  By this classification, each proper subgroup of $\PSL(2,p)$
  is a subgroup of a dihedral group,
  a group of affine type, which means that it is isomorphic
  to a subgroup of
  \[ \left\{ \begin{pmatrix}
            a & b \\
            0 & a^{-1}
          \end{pmatrix} 
          \bigm| 
          a\in \mathbb{F}^{*}, b\in \mathbb{F}
       \right\} / \{\pm 1\},
  \]
  or it is isomorphic to one of the groups
  $A_4$, $S_4$ or $A_5$.
  
  Since $p\equiv -1\mod 4 $ and $y\neq \pm 1$,
  we see that $s_1$ does not commute 
  with any  of $s_0$, $s_2$ and $s_0s_2$.
  It follows (as in~\cite{cherkassoffsjerve94})
  that $G=\erz{s_0,s_1,s_2}$ is not a subgroup of a dihedral group, 
  since in such a group we would
  have $\erz{s_0,s_2}\cap \Z(G)\neq \{1\}$.
  
  Since $C_2\times C_2\iso \erz{s_0,s_2}
                  \leq G$,
  the group can not be of affine type, either.
  Since $G$ contains an element of order $\geq 6$,
  the exceptional cases
  $G\iso A_4$, $S_4$ or $A_5$ are ruled out, too.
  Thus $G=\PSL(2,p)$, as claimed.
\end{proof}
\begin{lemma}\label{l:pslcomplexchar}
  If $p\equiv -1 \mod 4$,
  then there is $\chi\in \Irr(\PSL(2,p))$
  such that
  \begin{align*} 
   \chi(1)
      = \frac{p-1}{2}
     , \quad
     \chi(g) &\in \compl \setminus \reals
      \quad \text{if}\quad \ord(g)=p, \\
     \quad \text{and} \quad
     \chi(g) &\in \{-1,0,1\}
     \text{ else}.
  \end{align*}
  In particular, $\overline{\chi}\neq \chi$.
\end{lemma}
\begin{proof}
  We show this by using the Weil representation
  of $\SL(2,p)$, which equals the symplectic group in dimension
  $2$.
  The character $\psi$ of the Weil representation 
  has the property
  $\abs{\psi(g)}^2 = \abs{\Ker(g-1)}$ for all $g\in \SL(2,p)$,
  and decomposes into two irreducible characters
  $\psi = \psi_{+} + \psi_{-}$~\cite[Theorem~4.8]{i73}.
  (See also~\cite{howe73} and~\cite{prasad09p} 
  for an elementary approach to the Weil representation.)
  Here $\psi_{+}(-1) = \psi_{+}(1)$,
  so that the kernel of $\psi_{+}$ contains 
  $\{\pm 1\} = \Z(\SL(2,p))$
  and we can view $\chi= \psi_{+}$ as character of $\PSL(2,p)$.
  On the other hand, the constituent
  $\psi_{-}$ is defined by $\psi_{-}(-1) = -\psi_{-}(1)$.
  Thus we have $\psi(g)=\psi_{+}(g)+ \psi_{-}(g)$
  and $\psi(-g) = \psi_{+}(g) - \psi_{-}(g)$.
  It follows that 
  \[ \chi(g) = \psi_{+}(g) = \frac{1}{2} ( \psi(g)+\psi(-g) ).
  \]
  In particular, $\chi(1) = (p\pm 1)/2$.
  For our application this is actually all we need to know,
  but for completeness, let us mention that for 
  $p\equiv -1 \mod 4$ we have $\psi(-1)=-1$,
  so $\chi(1) = (p-1)/2$.
  (This follows from the known formulas for $\psi$~\cite{thomas08},
  but is easiest seen from remarking that $\psi_{-}(1)$
  must be even because $-1$ is in the kernel of the determinant of
  $\psi$.)
  
  If $g\in \SL(2,p)$ has order $p$,
  then 
  $\psi(g) = \pm \sqrt{-p}
  $~\cite[Corollary~6.2]{i73}\cite{thomas08},
  and $\psi(-g) = -1$.
  (Again, we only need to know that $\abs{\psi(-g)}=1$.)
  Therefore,
  $\chi(g) = (\pm \sqrt{-p} -1)/2$,
  and thus $\chi(g)\neq \overline{\chi(g)}$.
  
  If neither $g\in \SL(2,p)$ nor $-g$ has order $p$, then
  the order of $g$ is not divisible by $p$.
  In this case,
  $\psi(g)$ is rational~\cite[Proposition~2]{howe73}.
  Also, we have $\Ker(g-1) = \Ker(g+1) = \{0\}$, 
  except when $g=\pm 1$.
  It follows that
  $\psi(g)$, $\psi(-g) \in \{\pm 1 \}$.
  Thus 
  $\chi(g)= (1/2)(\pm 1 \pm 1)\in \{-1,0,1\}$.
\end{proof}
\begin{thm}\label{t:complbigwythoff}
  There are abstract regular polytopes 
  which have a pure realization of complex type
  with arbitrary large  essential Wythoff dimension.  
\end{thm}
\begin{proof}
  Let 
  $p$ be a prime such that $p\equiv -1\mod 4$ and
  $p\equiv 1 \mod 7$.
  Choose
  $y\in \mathbb{F}_p$ in Lemma~\ref{l:matrices_stringc}
  of multiplicative order $7$,
  and let $S_i$ and $s_i$ be as in Lemmas~\ref{l:matrices_stringc}
  and~\ref{l:pslgenerate}.
  Then $s_0s_1$ has order $7$.
  By these lemmas, $G=\PSL(2,p)$ is a string C-group
  with respect to $s_0$, $s_1$ and $s_2$.
  Thus there is an abstract regular polytope with vertex set
  the right cosets of $H=\erz{s_0,s_1}$.
  (Compared with Example~\ref{expl:pslconcrete}, the rôles 
   of $s_0$ and $s_2$ are now interchanged.)
  Notice that $H$ is a dihedral group of order $2\cdot 7=14$.
  
  Let $\chi$ be the character of Lemma~\ref{l:pslcomplexchar}
  and $S$ an irreducible module over $\reals G$ with character
  $\chi+\overline{\chi}$. 
  Then the essential Wythoff dimension of $S$ is
  \[ \ipcf{(1_H)^G, \chi}_G = \ipcf{1_H,\chi}_H
       \geq \frac{1}{14}\left( \frac{p-1}{2} - 13
                        \right)
       = \frac{p-1}{28}-\frac{13}{14}.
  \]
  Since there are infinitely primes
  $p$ with $p\equiv -1\mod 4$ and $p\equiv 1\mod 7$
  by Dirichlet's theorem,
  we can make this lower bound as large as we wish.
\end{proof}
The condition $p\equiv 1\mod 7$ in the proof was chosen only
for convenience. 
It is clear from the preceding lemmas that for 
``big'' primes $p$, we usually get a lot of 
possibilities of representing $\PSL(2,p)$
as a string C-group of type $\{k,l\}$,
with one or both of $k$, $l$ ``small''.

Checking small primes
suggests that every $\PSL(2,p)$, $19\leq p\equiv -1 \mod 4$,
is even a string C-group with respect to some 
generating set $\{s_0,s_1,s_2\}$ such that 
$s_0s_1$ has order $3$.

In~\cite[Remark~5.4]{mcmullen14},
McMullen says that he has ``not as
yet encountered any instances with 
[essential Wythoff dimension] $w^{*} > 2$''.
Of course, the examples of Theorem~\ref{t:complbigwythoff}
are such instances. 
However, another example is the $120$-cell.
As we mentioned in Example~\ref{expl:120cell}, 
there are two pure realizations of the $120$\nbd cell 
having Wythoff space of essential dimension $3$.

Even another example are the duals of the polytopes 
$\mathcal{L}_p^3$ with group 
$\PGL(2,p)$~\cite{mcmullen91,mcmullen11}.
The stabilizer of a facet of $\mathcal{L}_p^3$ has order $6$,
this is the stabilizer of a vertex of the dual polytope.
Since $\PGL(2,p)$ is $2$\nbd transitive on the $p+1$ lines
of $\mathbb{F}_p$ (in fact, sharply $3$-transitive),
the corresponding permutation character contains an irreducible 
character of degree $p$, which has values in $\{-1,0,1\}$ on
the non-identity elements of $\PGL(2,p)$.
The corresponding Wythoff space has dimension at least $(p-5)/6$.

\section{Orthogonality}
\label{sec:orthog}
On the set of matrices $\mat_{\Omega}(\reals)$,
the standard inner product is defined by
\[ \langle A, B \rangle
   =  \tr(AB^t).
\]
Now assume that $A = (a_{\xi \eta})$ and $B = (b_{\xi\eta})$ are 
$G$\nbd invariant matrices, and fix some $\alpha\in \Omega$.
Then for $\xi = \alpha g$ (say) we have
\[ \sum_{\eta \in \Omega} a_{\xi \eta}b_{\eta\xi}
           = \sum_{\eta \in \Omega} 
                a_{\alpha g, \eta}b_{\eta,\alpha g}
           = \sum_{\eta \in \Omega} 
                a_{\alpha g,\eta g}b_{\eta g,\alpha g}
           = \sum_{\eta \in \Omega} 
                a_{\alpha\eta}b_{\eta\alpha}.
\]
Thus
\[ \tr(AB^t)
   = \sum_{\xi,\eta\in\Omega} a_{\xi\eta}b_{\eta\xi}
   = \abs{\Omega} 
     \sum_{\eta\in \Omega} 
        a_{\alpha\eta}b_{\eta\alpha}.
\]
If additionally $A$ and $B$ are symmetric
(for example, $A$ and $B$ are inner product matrices
of realizations of $\Omega$),
then 
$\eta\mapsto a_{\alpha\eta}b_{\eta\alpha}$ 
is constant on the layers
of $\Omega$.
Let $\xi_0=\alpha$, $\xi_1$, $\dotsc$, $\xi_r$
be representatives of the layers
and define vectors $a$, $b\in \reals^{r+1}$
by $a_i = a_{\alpha, \xi_i}$, $b_i = b_{\alpha,\xi_i}$.
Let $\ell_i$ be the size of the layer containing $\xi_i$.
Then
\[ \tr(AB^t) 
   = \abs{\Omega} \sum_{\eta\in \Omega} 
                    a_{\alpha\eta}b_{\eta\alpha} 
   = \abs{\Omega} \sum_{i=0}^r \ell_i a_i b_i
   = \abs{\Omega}^2 \scp{a,b}_{\Lambda},
\]
where $\scp{a,b}_{\Lambda}$ is
the $\Lambda$\nbd inner product defined
by McMullen~\cite{mcmullen14} for inner product vectors.
So the correspondence between inner product vectors and 
inner product matrices identifies the $\Lambda$\nbd inner product
of McMullen with the standard inner product on matrices,
up to a scalar.
To maintain consistency with McMullen's notation, we write
\[ \lipr{A,B} = \frac{1}{\abs{\Omega}^2}\tr(AB^t)
\]
for $G$\nbd invariant, symmetric matrices $A$ and $B$.

\begin{thm}\label{t:ortho}
  If the simplex realization is written 
  as the blend of realizations
  $A_1 \oplus \dotsb \oplus A_s$,
  $A_i\colon \reals \Omega\to V_i$, with inner product matrices
  $Q_i$, 
  then 
  \[ \scp{Q_i,Q_j}_{\Lambda} 
      = \delta_{ij} 
         \frac{ \dim(V_i) }{ \abs{\Omega}^2 }. \]
\end{thm}
\begin{proof}
  The simplex realization is simply the identity
  $\id\colon \reals \Omega \to \reals \Omega$.
  The $A_i$ are then simply the orthogonal projections
  onto $V_i$, as are the $Q_i=A_iA_i^t=A_i^2=A_i$.
  It follows $Q_iQ_j=0$ for $i\neq j$,
  and $\tr(Q_i^2)= \tr(Q_i)=\dim V_i$.
\end{proof}
Notice that the $A_i$'s are not normalized realizations.
To normalize $A_i$, we have to scale $A_i$ by a factor
$\sqrt{\abs{\Omega}/\dim(V_i)}$.
So for the cosine matrices $C_i = \abs{\Omega}/\dim(V_i)$ 
of the $A_i$, we get
$\scp{C_i,C_i}_{\Lambda}= 1/\dim(V_i)$.
This is in accordance with~\cite[Theorem~4.5]{mcmullen14}.

The $\Lambda$-orthogonal basis of the realization cone
which McMullen constructs in~\cite{mcmullen14} is 
in general too small, due to the mistake in~\cite{mcmullenmonson03}.
We now indicate how to repair this.
We need to find orthogonal bases of the subcones
$\RC_{\sigma}(\Omega)$, for each $\sigma\in \Irr_{\reals}G$.
For this, we have to see what the isomorphism
of Theorem~\ref{t:homog_subconedim}
does to the scalar product.
Suppose that $A$ and $B\in \enmo_{\reals G}(\reals \Omega)$
are such that $e_{\sigma}A=A$ and $e_{\sigma}B=B$.
Choose $\mu_i$ and $\pi_i$ as in the proof
of Theorem~\ref{t:homog_subconedim},
and let $U=\reals \Omega e_{\sigma}$.
Then
\begin{align*}
  \tr_{\reals \Omega}(AB^t)
     &= \tr_U(AB^t)
     = \tr_U(\sum_i \pi_i\mu_i AB^t \sum_j \pi_j\mu_j)
     \\
     &= \sum_i \tr_S(\mu_i AB^t \pi_i)
     = \tr_S ( \sum_{i,j} a_{ij}\overline{b_{ij}} ),  
\end{align*}
where $a_{ij}=\mu_iA\pi_j \in \mathbb{D}$
and $\overline{b_{ij}} = (b_{ij})^t
       = (\mu_i B\pi_j)^t = \mu_j B^t \pi_i$.
Let $d = \sum_{i,j} a_{ij}\overline{(b_{ij})}
       = \tr( (a_{ij})(b_{ij})^*)$.
Then $\tr_S(d)= (\dim_{\reals}S) (d+\overline{d})/2$.

Thus the isomorphism of Theorem~\ref{t:homog_subconedim}
respects the canonical inner products on the involved spaces,
up to a scaling.
It is now clear how to choose an orthogonal basis
in the linear span of $\RC_{\sigma}(\Omega)$.
For example, if $m=2$ and $\mathbb{D}=\compl$,
we choose matrices corresponding to
\[ \begin{pmatrix}
      1 & 0 \\
      0 & 0
   \end{pmatrix}
   ,\quad 
   \begin{pmatrix}
      0 & 0 \\
      0 & 1
   \end{pmatrix}
   ,\quad 
   \begin{pmatrix}
      0 & 1 \\
      1 & 0
   \end{pmatrix}
   ,\quad 
   \begin{pmatrix}
      0 & i \\
      -i & 0
   \end{pmatrix}
\]
under the isomorphism of Theorem~\ref{t:homog_subconedim}.
Notice that the last two matrices do not correspond to 
realizations (they are not positive semi-definite).
Also, if $m>1$, the isomorphism of Theorem~\ref{t:homog_subconedim}
is by no means canonical, and thus we do not get a uniquely
defined basis.

\section{Cosine vectors and spherical functions}
\label{sec:spherical}
In this section, we explain the relation between 
cosine vectors and spherical functions, and use it to show
that the entries of a cosine vector of a realization with
essential Wythoff dimension $1$ are algebraic numbers.
We continue to assume that $G$ is a finite group,
$\Omega$ is a transitive $G$\nbd set 
and $H=G_{\alpha}$ is the stabilizer of some fixed initial vertex 
$\alpha$.
In the following, we set
\[ e_H := e_{1_H} = \frac{1}{\abs{H}} \sum_{h\in H} h
   .
\]
\begin{thm}\label{t:sumcosinechar}
  Let $S$ be a simple euclidean $G$\nbd space
  with character $\sigma$ 
  and with centralizer ring $\mathbb{D} = \enmo_{\reals G}(S)$.
  Let $W= \Fix_S(H)$ be the Wythoff space in $S$ and
  let $w_1$, $\dotsc$, $w_m$ be a basis of $W$ over $\mathbb{D}$
  such that the following hold: 
  We have $\scp{w_i,w_i}=1$,
  and whenever $i\neq j$ and $d_1$, $d_2\in \mathbb{D}$,
  then $\scp{w_id_1,w_jd_2}=0$.
  Then for all $g\in G$ we have
  \[ \sigma(e_Hg) = \ipcf{\sigma,\sigma}
                    \sum_{i=1}^m \scp{w_ig,w_i}.
  \]
\end{thm}
Before beginning with the proof, let us show how to construct 
a basis as in the theorem:
Begin with some $w_1\in W$ such that $\scp{w_1,w_1}=1$.
The orthogonal complement $U$ of 
$w_1 \mathbb{D}$ is closed under multiplication with $\mathbb{D}$,
since $\scp{ud,w_1}=\scp{u,w_1 \overline{d}}= 0$
for $u\in U$ and $d\in \mathbb{D}$.
By induction on the dimension, we find a basis in $U$ with the 
required properties, and thus one in $W$.

The case $m=1$ of the theorem 
is worth mentioning as a separate corollary:
\begin{cor}\label{c:cosine_spherical}
  Let $S$ be a simple euclidean $G$\nbd space with character 
  $\sigma$
  and essential Wythoff dimension $m=1$.
  Then for any $w\in W=\Fix_S(H)$ with
  $\scp{w,w}=1$ we have
  \[\scp{wg,w} = \frac{ \sigma(e_Hg) 
                     }{ \ipcf{\sigma,\sigma} 
                      }.
  \]
\end{cor}
Thus the cosine matrix of the corresponding pure realization
can be expressed in terms of the character of the 
corresponding irreducible representation.

To put Corollary~\ref{c:cosine_spherical} in perspective, 
we recall the notions of \emph{Gelfand pairs} and
\emph{spherical functions}.
(See~\cite[VII.1]{macdonald95} 
 or~\cite{ceccherini_etal08HAFG}
 for more on Gelfand pairs and spherical functions.)
Let $\pi$ be the permutation character of $G$ on $\Omega$
(we can think of $\Omega$ as the set of right cosets of $H$
in $G$ here).
The pair
$(G,H)$ is called a \emph{Gelfand pair},
if $\pi$ is multiplicity free (as $G$\nbd module over $\compl$),
that is, if 
$\ipcf{\pi,\chi}\leq 1$ for all  
$\chi\in \Irr(G)$.
(In our terminology, this is equivalent to
all essential Wythoff dimensions being $1$,
and the Wythoff dimensions itself are $1$
or $2$.)
If $\ipcf{\pi,\chi}=1$, then
the corresponding \emph{spherical function} $s_{\chi}$
is defined by
\[ s_{\chi}(g) = \chi(e_Hg) 
               = \frac{1}{\abs{H}}
                 \sum_{h\in H} \chi(hg).
\]
Thus Corollary~\ref{c:cosine_spherical} says that
if $S$ is of real type, then the entries of the corresponding 
cosine vector
are values of the spherical function $s_{\chi}$,
and if $S$ is of complex type, then the values of the 
cosine vector are the real parts of the spherical function.
It is well known that spherical functions can be expressed using
a $G$\nbd invariant inner product~\cite[VII (1.6)]{macdonald95}.

For example, it is a remarkable fact that the 
irreducible representations
of all finite Coxeter groups are of real type,
and it is another remarkable fact that the automorphism group
of almost every classical regular polytope
acts multiplicity freely on the vertices of the polytope;
the only exception is the $120$-cell.
In the other cases, the cosine vectors of the pure realizations
are thus the spherical functions.
These cosine vectors have been computed by 
McMullen~\cite{mcmullen89,mcmullen11,mcmullen14}.

Notice that when $\pi = (1_H)^G$ has a constituent $\sigma$
of quaternion type, then $(G,H)$ can not be a Gelfand pair,
since then $\sigma = 2\chi$ and $\ipcf{(1_H)^G,\chi}$
is a multiple of $2$.
We may say that $(G,H)$ is a Gelfand pair over $\reals$,
if $m_{\sigma}\in \{0,1\}$ for $\sigma\in \Irr_{\reals}G$,
that is, all essential Wythoff dimensions are $0$ or $1$.

\begin{proof}[Proof of Theorem~\ref{t:sumcosinechar}]
  Suppose $\overline{d}= -d$ for $d\in \mathbb{D}$.
  Then $\scp{vd,v}= \scp{v,v\overline{d}} = -\scp{v,vd}
        = -\scp{vd,v}$
  and thus $\scp{vd,v}=0$.
  We now choose a basis $B$ of $\mathbb{D}$ over $\reals$.
  If $\mathbb{D}=\reals$, we choose $B=\{1\}$,
  if $\mathbb{D}=\compl$, we choose $B=\{1,i\}$,
  and if $\mathbb{D}= \quats$, we choose
  $B=\{1,i,j,k\}$.
  In each case, it follows that
  $\scp{vb,vc}=0$ for $b\neq c\in B$
  and $\scp{vb,wb}=\scp{v,w}$.
  Thus $\{w_ib \mid i=1,\dotsc, m, b\in B\}$
  is an orthonormal basis of $W$ over $\reals$.
  Extend this basis by some set $X$ (say) to an orthonormal
  basis of the whole space $S$.
  For any $\reals$\nbd linear map $\alpha\colon S\to S$
  we have 
  \[ \tr(\alpha) = \sum_{i, b}\scp{w_ib\alpha,w_ib} 
                  + \sum_{x\in X}\scp{x\alpha,x}.
  \]
  We apply this to the map induced by $e_Hg$.
  Since $xe_H= 0$ for $x\not\in W$ and
  $we_H=w$ for $w\in W$, we get
  \begin{align*}
    \sigma(e_Hg)
       = \tr(e_Hg)
       = \sum_{i=1}^m \sum_{b\in B} \scp{w_ibe_Hg,w_ib}
      &= \sum_{i=1}^m \sum_{b\in B} \scp{w_igb,w_ib}
      \\
      &= \sum_{i=1}^m \sum_{b\in B} \scp{w_ig,w_i}
      \\
      &= \abs{B} \sum_{i=1}^m \scp{w_ig,w_i}
      \\
      &= \ipcf{\sigma,\sigma} \sum_{i=1}^m \scp{w_ig,w_i},
  \end{align*}
  as claimed.
\end{proof}

It follows from Corollary~\ref{c:cosine_spherical} that
the values of the cosine vector are algebraic numbers,
if $m=1$.
This confirms a ``guess'' of 
McMullen~\cite[Remark~9.4]{mcmullen14}.
We can say somewhat more:
It is known~\cite[VII(1.10)]{macdonald95} that 
$(\abs{HgH}/\abs{H})s_{\chi}(g)$ is an algebraic integer 
for spherical functions $s_{\chi}$.
We can extend this to the case where the essential Wythoff dimension
is $1$. 
\begin{cor}
  Let $S$ be an irreducible euclidean $G$\nbd space
  with essential Wythoff dimension $m=1$
  and let $w\in W= \Fix_S(H)$ have norm $1$. 
  Then 
  \[ \frac{\abs{HgH \cup Hg^{-1}H}}{\abs{H}} \scp{wg,w}
  \] 
  is an algebraic integer.
\end{cor}
Notice that $\abs{HgH\cup Hg^{-1}H}/\abs{H}$
is the size of the corresponding layer.
Another formulation of the corollary is thus: 
the component-wise product
of a cosine vector of a pure realization of essential
Wythoff dimension $1$ with the layer vector 
has algebraic integers as entries.
\begin{proof}
  For each double coset $K=HgH$, let
  \[ e_K = \frac{1}{\abs{H}} \sum_{x\in K} x \in \reals G.
  \]
  It is known~\cite[remarks before VII(1.10)]{macdonald95} 
  that the product of two such elements
  is a $\ints$\nbd linear combination of these elements.
  Thus $\ints[e_K \mid K \in \dblcs{H}{G}{H}]$
  is a ring which is finitely generated as $\ints$\nbd module,
  so its elements are integral.
  
  Let $W=Se_H \iso \mathbb{D}=\enmo_{\reals G}(S)$ 
  be the Wythoff space.
  Then $e_K= e_{HgH}$ acts as some 
  $\mathbb{D}$\nbd linear map 
  on $W$, and can thus be identified with some
  $d\in \mathbb{D}$.
  Then $e_K + e_{K^{-1}} = e_{HgH} + e_{Hg^{-1}H}$ acts as the
  scalar $\lambda = d + \overline{d}$ on $W$.
  Since $e_K$ is integral over $\ints$,
  it follows that $d$ and $\lambda$ are integral over $\ints$.
  In the case where $d\in \reals$ we have
  \[d = \frac{\sigma(e_K) }{ \ipcf{\sigma,\sigma} }
     =  \frac{ \sigma(e_He_K)}{ \ipcf{\sigma,\sigma} }
   =  \frac{1}{\abs{H}} \sum_{x\in K} \scp{wx,w}
   = \frac{\abs{K}}{\abs{H}}\scp{wg,w},
  \]
   and in any case we have 
   \[ \lambda = \frac{ \sigma(e_K + e_{K^{-1}}) 
                     }{ \ipcf{\sigma,\sigma} }
              = 2 \frac{\sigma(e_K)}{ \ipcf{\sigma,\sigma} }
              = 2 \frac{\abs{K}}{\abs{H}} \scp{wg,w}.
   \]
   Notice that if $K$ is symmetric, then 
   necessarily $d\in \reals$.
  The result follows.  
\end{proof}

\section{On the realizations of the 600-cell}
\label{sec:600cell}
In this section we explain two observations of 
McMullen~\cite[Remark~9.3]{mcmullen11}
about the pure realizations of the 600-cell.
Namely, we have the following:
\begin{thm}\label{t:600cell}
  There is a ``natural'' bijection
  between the irreducible characters
  of the finite group $\SL(2,5)$
  and the pure realizations of the $600$\nbd cell.
  If $\phi\in \Irr(\SL(2,5))$, then the corresponding pure
  realization has dimension $\phi(1)^2$,
  and the entries of its cosine vector are of the
  form $\phi(u)/\phi(1)$, where $u$ runs through
  $\SL(2,5)$.
  (More precisely, we also have a natural bijection between
   the conjugacy classes of $\SL(2,5)$ and the layers
   of the $600$\nbd cell, 
   and $\phi(u)/\phi(1)$ is the value at the layer
   corresponding to the conjugacy class of $u$.)
\end{thm}
This ``explains'' that the dimension of each pure realization
is a square $q^2$, and that its cosine vector has entries
of the form $a/q$, where $a$ is an algebraic integer
(in fact, $a\in \ints[\tau]$ with $\tau= (-1+\sqrt{5})/2$).

We have to warn the reader that the proof of 
Theorem~\ref{t:600cell},
while not difficult, is rather long,
in particular longer than working out the cosine vectors directly.
On the other hand, we work out the realization cone of a 
class of $G$-sets, of which the $600$\nbd cell is an example.

We will use that the automorphism group of the $600$\nbd cell,
the reflection group of type $\mathrm{H}_4$,
is the factor group of a certain wreath product:
Let $U$ be a group.
The cyclic group $C_2=\{1,t\}$ of order $2$ acts on the
direct product $U\times U$ be exchanging components, that is
$(u,v)^{t} = (v,u)$.
The corresponding semidirect product
of $C_2$ and $U\times U$ is the wreath product,
denoted by $U \wr C_2$.
The following lemma is of course known, but for completeness,
 we work out a large part of the proof:
\begin{lemma}
  Set $U=\SL(2,5)$ and
  $\widehat{G} = U\wr C_2$, 
  and let $\widehat{H}$ be the subgroup of 
  $\widehat{G}$ generated by 
  the pairs $\{(u,u) \mid u\in U\}$
  and by $C_2$.
  (Notice that $\widehat{H}\iso C_2 \times U$.)
  The automorphism group of the 600-cell is isomorphic 
  to the factor group $\widehat{G}/\Z(\widehat{G})$
  in such a way that the
  stabilizer of a vertex is identified with 
  $\widehat{H}/\Z(\widehat{G})$.  
\end{lemma}
\begin{proof}
We can express the automorphism group of the 
$600$\nbd cell as a group of transformations on the
quaternions $\quats$.
For $u\in \quats$, let
$\lambda_u\colon \quats\to \quats$ and
$\rho_u\colon \quats \to \quats$ be the maps defined by
\[ x \lambda_u = \overline{u}x
   \quad \text{and} \quad
   x\rho_u = xu
   \quad (x\in \quats).
\]
Let $\sigma\colon \quats \to \quats$ be conjugation.

Let $U$ be a (finite) subgroup of the multiplicative
group $\quats^{*}$.
Mapping $t$ to $\sigma$ and 
$(u,v)\in U\times U$ to $\lambda_{u}\rho_u$
defines a group homomorphism from $U \wr C_2$
into $\GL_{\reals}(\quats)\iso \GL(4,\reals)$.
The kernel is $\erz{(-1,-1)}\subseteq U\times U$.

The reflection group of type $\mathrm{H}_4$ can be realized
as the image of such a homomorphism:
Let 
\[ \alpha_1 = j ,\quad
   \alpha_2 = \frac{1}{2}(ai+bj-k), \quad
   \alpha_3 = k, \quad 
   \alpha_4 = \frac{1}{2}(a + bi -k),
\]
where $a = 2\cos(2\pi/5) = (-1+\sqrt{5})/2$
and $b = 2\cos(4\pi/5) = (-1-\sqrt{5})/2$.
Then $\alpha_1$, $\dotsc$, $\alpha_4$ form a simple root
system of type $\mathrm{H}_4$.

Let $s_1$, $\dotsc$, $s_4$ be the reflections corresponding to 
$\alpha_1$, $\dotsc$, $\alpha_4$.
These generate the automorphism group $G$ of the $600$\nbd cell,
and the stabilizer of a vertex is $H=\erz{s_1,s_2,s_3}$.
(The vertices are all points in the orbit of $1=1_{\quats}$.)

The reflection corresponding to an element 
$\alpha\in \quats$ of norm $1$ 
is the map 
\[ x \mapsto - \alpha \overline{x} \alpha 
     = x\sigma \lambda_{-\overline{\alpha}}\rho_{\alpha}, 
\]
as is easily checked
(it sends $\alpha$ to $-\alpha$ and fixes
$i\alpha$, $j\alpha$ and $k\alpha$).
It follows that
\[ \erz{s_1,s_2,s_3,s_4}
   \subseteq \{\id_{\quats},\sigma\}
             \{ \lambda_u\rho_v \mid u, v \in U
             \},
\]
where $U$ is the group generated by 
$\alpha_1$, $\dotsc$, $\alpha_4$ and $-1$.

Since $\alpha_1^2=-1$ and $\alpha_4= (\alpha_1\alpha_2)^2$,
we see that 
$U=\erz{\alpha_1,\alpha_2,\alpha_3}$.
We see that the reflections $s_1$, $s_2$, $s_3$ generate the 
subgroup
\[ H = \{\id_{\quats},\sigma\}
                 \{\lambda_{u}\rho_u\mid u \in U\}
      .
\]
Then it is also not difficult to see that
\[\erz{s_1,s_2,s_3,s_4}
  = \{\id_{\quats},\sigma\}\{\lambda_u\rho_v\mid u, v \in U\}.
\]
We leave out the proof that $U\iso \SL(2,5)$.
Apart from this, the lemma is proved.
\end{proof}

We now slightly change notation.
Let $U$ be an arbitrary finite group,
let $G$ be the wreath product
$U \wr C_2$ and let $H\leq G$ be the subgroup
\[ H = \{1,t\}\{(u,u)\mid u\in U \} 
     \iso C_2 \times U.
\]
We will describe the realization cone of the $G$\nbd set
$[G:H]$ (the cosets of $H$ in $G$) for such $G$ and $H$.

Set $N= U\times U$, a normal subgroup of
$G$ of index $2$.
The irreducible characters of $N$ are of the form
$\phi\times \theta$ with
$\phi$, $\theta\in \Irr U$~\cite[Theorem~III.9.1]{simonreps}.
\begin{lemma}
  \begin{enumthm}
  \item If $\phi\neq \theta\in \Irr U$, then
        $(\phi\times \theta)^G \in \Irr G$.
  \item For $\phi\in \Irr U$, the character 
        $\phi\times \phi$ 
        has exactly two extensions to a character of $G$,
        namely 
        \[ \chi(t(u,v))
           = \phi(uv)
           \quad\text{and}\quad
           \chi(t(u,v))
           = -\phi(uv).\]
  \end{enumthm}
\end{lemma}
\begin{proof}
  The first point is clear from Clifford theory
  ($(\phi\times \theta)^G$ denotes the Frobenius induced character).
  
  It is also known that $\phi\times \phi$ has two different 
  extensions
  to $G$~\cite[III.11]{simonreps}. 
  Here, we can describe these extensions explicitly.
  Let $X$ be a $\compl U$\nbd module affording 
  the character $\phi$.
  We may define an action of $t$ on $X\otimes X$ by
  $(x\otimes y)t = y \otimes x$
  or $(x\otimes y)t = -y\otimes x$.
  These are the two extensions to a representation of $G$.
  
  We treat the first case. Then
  \[ (x\times y)t(u,v) = yu\otimes xv.\]
  Suppose that $\{e_i\}$ is a basis of $X$
  and $e_i u = \sum_{j} d_{ij}(u)e_j$.
  Then $\{e_i\otimes e_j\}$ is a basis of $X\otimes X$,
  and we get for the trace of $t(u,v)$ on $X\otimes X$:
  \begin{align*}
    \chi(t(u,v)) 
      &= \sum_{i,j} d_{ji}(u)d_{ij}(v)
       = \sum_j d_{jj}(uv)
       = \phi(uv).\qedhere
  \end{align*} 
\end{proof}
Let $U$, $G$, $H$ and $N$ be as defined before the last lemma.
\begin{lemma}\label{l:constpermchar}
  If $\chi\in \Irr G$ with
  $\ipcf{\chi_H, 1}\neq 0$, then either
  $\chi = (\phi\times \overline{\phi})^G$
  with $\phi\neq \overline{\phi}\in \Irr U$,
  or $\chi_N = \phi\times \phi$
  with $\phi=\overline{\phi}\in \Irr U$ and
  $\chi(\sigma(u,v)) = \nu_2(\phi)\phi(uv)$.
  In both cases, $[\chi_H,1] = 1$.
\end{lemma}
(Here $\nu_2(\phi)$ denotes the Frobenius-Schur indicator of 
$\phi$.
Recall that for $\phi\in \Irr U$,
\[ \nu_2(\phi) = \frac{1}{2\abs{U}} \sum_{u\in  U} \phi(u^2)
   \in \{0,\pm1\},
\]
and $\nu_2(\phi)\neq 0 $ if and only if 
$\phi=\overline{\phi}$~\cite[Theorem~III.5.1]{simonreps}.)

Lemma~\ref{l:constpermchar} 
explains the first part of Theorem~\ref{t:600cell}.
Since $U=\SL(2,5)$ has only real-valued characters,
every pure realization corresponds to a
$\phi\in \Irr U$ and has dimension $\phi(1)^2$.

In the general case, notice that the realizations correspond to
$\Irr_{\reals} U$. 
The Wythoff dimension is $1$ for all pure realizations.
(In particular, the corresponding irreducible representations
 are of real type.)
Thus the realization cone is polyhedral, in fact a direct product
of copies of $\reals_{\geq 0}$ by Theorem~\ref{t:homog_subconedim}.
\begin{proof}[Proof of Lemma~\ref{l:constpermchar}]
  First, suppose that 
  $\chi = (\phi\times \theta)^G$ with $\phi\neq \theta\in \Irr U$.
  Then
  \begin{align*}
    \ipcf{\chi_H,1_H}
      = \ipcf*{\big((\phi\times\theta)^G\big)_H,1_H}
      & = \ipcf*{\big((\phi\times\theta)_{H\cap N}\big)^H,1_H}
      \\
      &= \ipcf{(\phi\times\theta)_{H\cap N},1_{H\cap N}}
       \\
      &= \frac{1}{\abs{U}}
         \sum_{u\in U} \phi(u)\theta(u)
      \\   
      &= \ipcf{\phi, \overline{\theta}}_U
       = \delta_{\phi,\overline{\theta}}.
  \end{align*}
  Here the second equality follows from $G=HN$ and 
  Mackey's formula, and the third equality follows from
  Frobenius reciprocity.
  Thus $\theta = \overline{\phi}\neq \phi$ when 
  $\ipcf{\chi_H,1_H}\neq0$.
  
  Second, suppose that $\chi$ extends $\phi\times \phi$,
  and that $\chi(t(u,v))= \eps \phi(uv)$.
  Then
  \begin{align*}
    \ipcf{\chi_H,1_H}
      &= \frac{1}{2\abs{U}}
         \sum_{u\in U} \big(\chi((u,u)) + \chi(t(u,u))\big)
      \\&= \frac{1}{2\abs{U}} \left( \sum_{u\in U} \phi(u)^2
           + \sum_{u\in U} \eps \phi(u^2) \right)
        = \frac{1}{2}( \ipcf{\phi, \overline{\phi}} 
                       + \eps \nu_2(\phi) ).
  \end{align*}
  The last expression is non-zero only when
  $\phi=\overline{\phi}$ and
  $\eps = \nu_2(\phi)$, and in this case $\ipcf{\chi_H,1}=1$.
\end{proof}
The next result finishes the proof of Theorem~\ref{t:600cell}.
As in the last results, we only assume that $G= U \wr C_2$
for some finite group $U$, and that
$H=C_2\{(u,u)\mid u\in U\}$.
We notice in passing that in this situation, 
\[ Ht(x,y)H = H(x,y)H \leftrightarrow 
   (x^{-1}y)^U \cup (y^{-1}x)^U
\]
defines a bijection between double cosets of $H$ and 
``symmetrized'' conjugacy classes of $U$.
The double cosets of $H$ in turn correspond to the layers.
(If $U=\SL(2,5)$, then all conjugacy classes of $U$ 
 are real, that is, $u$ and $u^{-1}$ are always conjugate.)
The following lemma describes an arbitrary entry of a cosine vector
of a pure realization.
\begin{lemma}
  Let $V$ be an irreducible euclidean $G$\nbd space
  and suppose the non-zero element 
  $w\in V$ is fixed by $H$.
  Then the character $\chi$ of $V$ is irreducible.
  Let $\phi\in \Irr U$ be the character defined in
  Lemma~\ref{l:constpermchar}.
  Let $n=(x,y)\in N=U\times U$.
  Then 
  \[ \frac{\scp{wn,w}}{\scp{w,w}} =
      \frac{\phi(x^{-1}y)}{\phi(1)}.
  \]
\end{lemma}
\begin{proof}
  Since $w\neq 0$ is fixed by $H$, we have
  $\ipcf{\chi_H, 1_H}\neq 0$. 
  It follows from Lemma~\ref{l:constpermchar} that
  $\ipcf{\chi_H,1_H}=1$, and $\chi$ is as in that lemma.
  We may assume that $\scp{w,w}=1$ and apply 
  Corollary~\ref{c:cosine_spherical}.
  We only treat the case that
  $\chi_N = \phi\times \phi$.
  (The case $\chi = (\phi\times \overline{\phi})^G$ 
  is similar, but in fact simpler.)
  We get
  \begin{align*}
    \scp{wn,w}
      = \chi(e_H n)
      &= \frac{1}{2\abs{U}} 
        \left( \sum_{u\in U} \chi((ux,uy)) 
               +\sum_{u\in U} \chi(t(ux,uy)) 
        \right)
     \\ &= \frac{1}{2\abs{U}}
           \left( \sum_{u\in U} \phi(ux)\phi(uy)
                 + \sum_{u\in U} \nu_2(\phi)\phi(uxuv)
           \right).
  \end{align*}
  The first sum equals $\abs{U}\phi(x^{-1}y)/\phi(1)$
  by the generalized orthogonality 
  relation~\cite[Theorem~2.13]{isaCTdov}
  and the fact that 
  $\phi(uy)=\overline{\phi(uy)}=\phi(y^{-1}u^{-1})$.
  For the second sum, we get
  \[ \frac{1}{\abs{U}}
     \sum_{u\in U} \phi(uxuy)
     = \frac{1}{\abs{U}} \phi\left( \sum_{v\in U} v^2 x^{-1}y \right)
     = \phi( z x^{-1}y),
  \]
  where $z= (1/\abs{U}) \sum_{v\in U} v^2$
  is a central element in the group algebra
  and is mapped to a scalar matrix by any irreducible representation.
  Thus $\phi(zx^{-1}y) = (\phi(z)/\phi(1))\phi(x^{-1}y)$.
  But clearly, $\phi(z)= \nu_2(\phi)$.
  Plugging in above, we get that
  $\chi(e_H g) = \phi(x^{-1}y)/\phi(1)$ as claimed.
\end{proof}

\printbibliography


\end{document}